\documentclass[12pt,reqno]{amsart}
\usepackage[margin=1in]{geometry}
\usepackage{amssymb,amsfonts,amsmath,mathrsfs,etoolbox,mathtools,bm}
\usepackage[breaklinks=true,colorlinks=true,linkcolor=black!25!cyan,citecolor=black!25!cyan,filecolor=cyan,urlcolor=magenta]{hyperref}
\usepackage[hiresbb]{graphicx,xcolor}
\allowdisplaybreaks[4]

\numberwithin{equation}{section}

\newtheorem{theorem}{Theorem}[section]
\newtheorem{corollary}[theorem]{Corollary}
\newtheorem{proposition}[theorem]{Proposition}

\theoremstyle{remark}
\newtheorem{remark}{Remark}

\newcommand{\Z}{\mathbb{Z}}

\newcommand{\R}{\mathbb{R}}

\newcommand{\bp}{\bm{p}}
\newcommand{\PSL}{\mathrm{PSL}}
\newcommand{\tr}{\mathrm{tr}}
\newcommand{\norm}[1]{\left\lVert#1\right\rVert}

\renewcommand{\epsilon}{\varepsilon}

\renewcommand{\Re}{\mathrm{Re}}
\renewcommand{\tilde}{\widetilde}

\renewcommand{\pmod}[1]{\, (\mathrm{mod} {\, #1})}

\patchcmd{\section}{\scshape}{\bfseries}{}{}
\makeatletter\renewcommand{\@secnumfont}{\bfseries}\makeatother

\begin{document}

\title{The Prime Geodesic Theorem in Arithmetic Progressions}

\author{Dimitrios Chatzakos}
\address{Department of Mathematics, University of Patras, 26 110, Patras, Greece}
\urladdr{\href{https://dchatzakos.math.upatras.gr}{https://dchatzakos.math.upatras.gr}}
\email{dimitris.xatzakos@gmail.com}

\author{Gergely Harcos}
\address{Alfr\'{e}d R\'{e}nyi Institute of Mathematics, POB 127, Budapest H-1364, Hungary}
\urladdr{\href{https://users.renyi.hu/~gharcos/}{https://users.renyi.hu/~gharcos/}}
\email{gharcos@renyi.hu}

\author{Ikuya Kaneko}
\address{The Division of Physics, Mathematics and Astronomy, California Institute of Technology, 1200 E. California Blvd., Pasadena, CA 91125, USA}
\urladdr{\href{https://sites.google.com/view/ikuyakaneko/}{https://sites.google.com/view/ikuyakaneko/}}
\email{ikuyak@icloud.com}

\thanks{The first author acknowledges the financial support from the ELKE of the University of Patras (MEDIKOS Program no. 82043). The second author was supported by the R\'enyi Int\'ezet Lend\"ulet Automorphic Research Group and NKFIH (National Research, Development and Innovation Office) grant K~143876. The third author acknowledges the support of the Masason Foundation.}

\subjclass[2020]{11F72 (primary); 11M36 (secondary)}

\keywords{prime geodesic theorem, arithmetic progressions, Kuznetsov--Bykovski\u{\i} formula}

\date{\today}

\dedicatory{}

\begin{abstract}
We address the prime geodesic theorem in arithmetic progressions, and resolve conjectures of Golovchanski\u{\i}--Smotrov (1999). In particular, we prove that the traces of closed geodesics on the modular surface do not equidistribute in the reduced residue classes of a given modulus.
\end{abstract}

\maketitle

\section{Introduction}\label{introduction}

\subsection{Historical prelude}\label{historical-prelude}
The prime geodesic theorem asks for an asymptotic evaluation~of the counting function of oriented primitive closed geodesics on hyperbolic manifolds. If the underlying group is a cofinite Fuchsian group $\Gamma \subset \PSL_{2}(\R)$, then this problem has received distinguished attention among number theorists; see for instance~\cite{Hejhal1976-2,Hejhal1976,Hejhal1983,Huber1961,Huber1961-2,Kuznetsov1978,Sarnak1980,Selberg1989,Venkov1990} and references therein for classical triumphs.

A closed geodesic $P$ on $\Gamma \backslash \mathbb{H}$ corresponds bijectively to a hyperbolic conjugacy class in $\Gamma$ (cf.~\cite{Huber1959}). If $\mathrm{N}(P)$ denotes the norm of this conjugacy class, then the hyperbolic length of $P$ equals $\log\mathrm{N}(P)$. As usual, we set $\Lambda_{\Gamma}(P) = \log \mathrm{N}(P_{0})$, where $P_0$ is the primitive closed geodesic underlying $P$, and we introduce the Chebyshev-like counting function
\begin{equation*}
\Psi_{\Gamma}(x) \coloneqq \sum_{\mathrm{N}(P) \leq x} \Lambda_{\Gamma}(P).
\end{equation*}
Because the prime geodesic theorem is reminiscent of the prime number theorem, the norms are sometimes called pseudoprimes. Selberg~\cite{Selberg1989} established an asymptotic formula of the shape
\begin{equation}\label{eq:Selberg}
\Psi_{\Gamma}(x) = \sum_{\frac{1}{2} < s_{j} \leq 1} \frac{x^{s_{j}}}{s_{j}}+\mathcal{E}_{\Gamma}(x),
\end{equation}
where the main term emerges from the small eigenvalues $\lambda_{j} = s_{j}(1-s_{j}) < \frac{1}{4}$ of the Laplacian on $\Gamma \backslash \mathbb{H}$ for the upper half-plane $\mathbb{H}$, and $\mathcal{E}_{\Gamma}(x)$ is an error term. It is known that $\mathcal{E}_{\Gamma}(x) \ll_{\Gamma,\epsilon} x^{\frac{3}{4}+\epsilon}$; see the explicit formul{\ae} in~\cite{Iwaniec1984,KanekoKoyama2022}. This barrier is often termed the trivial bound. Given an analogue of the Riemann Hypothesis for Selberg zeta functions apart from a finite number of the exceptional zeros, one should expect the best possible estimate to be $\mathcal{E}_{\Gamma}(x) \ll_{\Gamma,\epsilon} x^{\frac{1}{2}+\epsilon}$. This remains unresolved due to the abundance of Laplace eigenvalues.

If $\Gamma$ is arithmetic, then an improvement over the barrier was achieved by Iwaniec~\cite{Iwaniec1984}, who proved that $\mathcal{E}_{\Gamma}(x) \ll_{\epsilon} x^{\frac{35}{48}+\epsilon}$ for the full modular group $\Gamma = \PSL_{2}(\Z)$. Iwaniec~\cite{Iwaniec1984-2} stated that the stronger exponent $\frac{2}{3}+\epsilon$ follows from the Generalised Lindel\"{o}f Hypothesis~for quadratic Dirichlet $L$-functions. Furthermore, under the Generalised Lindel\"{o}f Hypothesis for Rankin--Selberg $L$-functions, the same exponent follows from the Weil bound for Kloosterman sums. Luo--Sarnak~\cite{LuoSarnak1995} then strengthened the machinery of Iwaniec to obtain the exponent $\frac{7}{10}+\epsilon$; see also~\cite{Koyama1998,LuoRudnickSarnak1995}. As a further refinement, Cai~\cite{Cai2002} derived the exponent $\frac{71}{102}+\epsilon$. The crucial gist in all these works is to estimate nontrivially a certain spectral exponential sum via the Kuznetsov formula. On the other hand, the subsequent work of Soundararajan--Young~\cite{SoundararajanYoung2013} demonstrated that
\begin{equation*}
\mathcal{E}_{\Gamma}(x) \ll_{\epsilon} x^{\frac{2}{3}+\frac{\vartheta}{6}+\epsilon},
\end{equation*}
where $\vartheta$ is a subconvex exponent for quadratic Dirichlet $L$-functions. The current record~$\vartheta = \frac{1}{6}$ by Conrey--Iwaniec~\cite{ConreyIwaniec2000} implies the best known exponent $\frac{25}{36}+\epsilon$. The proof utilises the Kuznetsov--Bykovski\u{\i} formula (see~\cite{Bykovskii1994,Kuznetsov1978,SoundararajanYoung2013}), leaving the theory of Selberg zeta functions aside. For recent progress~on~the prime geodesic theorem and its generalisations, we direct the reader to~\cite{BalogBiroCherubiniLaaksonen2022,BalogBiroHarcosMaga2019,BalkanovaChatzakosCherubiniFrolenkovLaaksonen2019,BalkanovaFrolenkov2019,BalkanovaFrolenkov2020,BalkanovaFrolenkovRisager2022,ChatzakosCherubiniLaaksonen2022,CherubiniGuerreiro2018,CherubiniWuZabradi2022,DeverMilicevic2023,Kaneko2020,Kaneko2022-2,Kaneko2023-2,Koyama2001,PetridisRisager2017}.

\subsection{Statement of main results}\label{statement-of-main-results}
By a classical theorem of Dirichlet, the primes equidistribute in the reduced residue classes of a given modulus. As we shall see, the prime geodesic analogue of this phenomenon breaks down, and the corresponding non-uniform distribution can be determined explicitly.

We fix $\Gamma = \PSL_{2}(\Z)$. By definition, an element $P \in \Gamma$ is said to be hyperbolic if as a M\"{o}bius transformation
\begin{equation*}
Pz = \frac{az+b}{cz+d},
\end{equation*}
it possesses two distinct real fixed points. By conjugation, any hyperbolic element $P\in\Gamma$ may be expressed as $P = \sigma^{-1} \tilde{P} \sigma$, where $\sigma \in \PSL_{2}(\R)$ and $\tilde{P} = \begin{pmatrix} \lambda & 0 \\ 0 & \lambda^{-1} \end{pmatrix}$ with $\lambda > 1$. Here $\tilde{P}$ acts as multiplication by $\lambda^{2}$, namely $\tilde{P} z = \lambda^{2} z$. The factor $\lambda^{2}$ is called the norm of $P$, which depends only on the $\PSL_{2}(\R)$-conjugacy class of $P$. We note that the positive trace
\begin{equation*}
t = \tr(P) = \tr(\tilde{P}) = \lambda+\lambda^{-1}
\end{equation*}
is an integer exceeding $2$, and hence the norm takes the form
\begin{equation*}
\mathrm{N}(P) = \left(\frac{t+\sqrt{t^{2}-4}}{2} \right)^{2} = t^{2}-2+O(t^{-2}).
\end{equation*}
Given a prime $p \geq 2$, we define the Chebyshev-like counting function in arithmetic progressions modulo $p$ by
\begin{equation*}
\Psi_{\Gamma}(x; p, a) \coloneqq \sum_{\substack{\mathrm{N}(P) \leq x \\ \tr(P) \equiv a \pmod p}} \Lambda_{\Gamma}(P).
\end{equation*}

The following result shows that the main term in an asymptotic for $\Psi_{\Gamma}(x; p, a)$ depends on the residue class $a \pmod{p}$ unlike for primes in arithmetic progressions.

\begin{theorem}\label{main}
Let $\Gamma = \PSL_{2}(\Z)$, and let $p \geq 3$ be a prime. Then we have that
\begin{equation}\label{eq:asymptotic}
\Psi_{\Gamma}(x; p, a) = 
\begin{dcases}
\frac{1}{p-1} \cdot x+O_{\epsilon}(x^{\frac{3}{4}+\frac{\vartheta}{2}+\epsilon}) & \text{if $\left(\frac{a^{2}-4}{p} \right) = 1$,}\\
\frac{1}{p+1} \cdot x+O_{\epsilon}(x^{\frac{3}{4}+\frac{\vartheta}{2}+\epsilon}) & \text{if $\left(\frac{a^{2}-4}{p} \right) = -1$,}\\
\frac{p}{p^{2}-1} \cdot x+O_{p,\epsilon}(x^{\frac{3}{4}+\frac{\vartheta}{2}+\epsilon}) & \text{if $\left(\frac{a^{2}-4}{p} \right) = 0$,}
\end{dcases}
\end{equation}
where $\vartheta$ is a subconvex exponent for quadratic Dirichlet $L$-functions.
\end{theorem}

\begin{remark}
The implied constant in the error term is independent of $p$ when $a\not\equiv\pm 2\pmod{p}$.
\end{remark}

\begin{remark}When $p = 3$, the first case of \eqref{eq:asymptotic} is void, while the second case is covered with a stronger error term by~\cite[Theorem~$1$]{GolovchanskiiSmotrov1999}.
\end{remark}

\begin{remark}
Apart from the size of the error term, Theorem~\ref{main} resolves \cite[Conjecture~$2$]{GolovchanskiiSmotrov1999} for level $N=1$. In fact, we expect that our method works for $\Gamma=\Gamma_0(N)$ when $(N, p) = 1$, but we restrict to $\Gamma = \PSL_{2}(\Z)$ for simplicity. Furthermore, we expect that the error term can be improved significantly by a more careful analysis (e.g.\ by combining the Kuznetsov--Bykovski\u{\i} formula with an ad\`{e}lic trace formula), but we solely focused on determining the main term according to the sign of the Legendre symbol. We leave such pursuits for future work.
\end{remark}

The method of Golovchanski\u{\i}--Smotrov~\cite[Theorem~$1$]{GolovchanskiiSmotrov1999} is different from ours, and they delve into properties of the norms and traces, expressing $\Psi_{\Gamma}(x; p, a)$ as a linear combination of $\Psi_{\Gamma_{0}(2^{k})}(x)$ for some $k \geq 0$, for which an asymptotic formula is already known as in~\eqref{eq:Selberg}. For example, they derived a general linear combination of the shape
\begin{equation*}
3 \Psi_{\Gamma_{0}(N)}(x)-3 \Psi_{\Gamma_{0}(2N)}(x)+\Psi_{\Gamma_{0}(4N)}(x) = 3 \Psi_{\Gamma_{0}(N)}(x; 2, 1),
\end{equation*}
from which it follows that\footnote{We emphasise that the case of $p = 2$ is not contained in Theorem~\ref{main}.}
\begin{equation*}
\Psi_{\Gamma_{0}(N)}(x; 2, 1) = \frac{1}{3} \cdot x+\mathcal{E}_{\Gamma_{0}(N)}(x), \qquad 
\Psi_{\Gamma_{0}(N)}(x; 2, 0) = \frac{2}{3} \cdot x+\mathcal{E}_{\Gamma_{0}(N)}(x).
\end{equation*}
The level structure that they developed is delicate, and it appears that their idea only works for some specific values of $p$ and $a$. Hence, some new ideas are needed to prove Theorem~\ref{main}.

An elementary counting argument implies the following result. 

\begin{corollary}\label{corollary}
Let $\Gamma = \PSL_{2}(\Z)$, and let $p \geq 3$ be a prime. Then we have that
\begin{align*}
\sum_{\substack{a \pmod{p} \\ \left(\frac{a^{2}-4}{p} \right) = 1}} \Psi_{\Gamma}(x; p, a)
& = \frac{p-3}{2(p-1)} \cdot x+O_{p,\epsilon}(x^{\frac{3}{4}+\frac{\vartheta}{2}+\epsilon}),\\
\sum_{\substack{a \pmod{p} \\ \left(\frac{a^{2}-4}{p} \right) = -1}} \Psi_{\Gamma}(x; p, a)
& = \frac{p-1}{2(p+1)} \cdot x+O_{p,\epsilon}(x^{\frac{3}{4}+\frac{\vartheta}{2}+\epsilon}),\\
\sum_{\substack{a \pmod{p} \\ \left(\frac{a^{2}-4}{p} \right) = 0}} \Psi_{\Gamma}(x; p, a)
& = \frac{2p}{p^{2}-1} \cdot x+O_{p,\epsilon}(x^{\frac{3}{4}+\frac{\vartheta}{2}+\epsilon}),
\end{align*}
where $\vartheta$ is a subconvex exponent for quadratic Dirichlet $L$-functions.
\end{corollary}

\begin{remark}\label{rem:combination}
Apart from the size of the error term, Corollary~\ref{corollary} resolves~\cite[Conjecture~$1$]{GolovchanskiiSmotrov1999} in the case of full level $N=1$, and again the method should work for $\Gamma = \Gamma_{0}(N)$ when $(N, p) = 1$.
\end{remark}

\subsection*{Acknowledgements}
The authors thank Mikhail Nikolaevich Smotrov for sending us the preprint~\cite{GolovchanskiiSmotrov1999}.

\section{Key propositions}\label{key-propositions}
This section prepares for the proof of Theorem~\ref{main}. Throughout, we follow \cite[Section~2]{SoundararajanYoung2013} closely.

Let $\Gamma=\PSL_{2}(\Z)$ as before. Sarnak~\cite[Proposition~1.4]{Sarnak1982} showed that the primitive hyperbolic conjugacy classes in $\Gamma$ correspond bijectively to the $\Gamma$-equivalence classes of primitive indefinite binary quadratic forms. We recall this correspondence briefly. For a given primitive quadratic form $ax^{2}+bxy+cy^{2}$ of discriminant $d>0$, the automorphs are the elements
\begin{equation*}
P(t, u) = \begin{pmatrix} \dfrac{t-bu}{2} & -cu \\ au & \dfrac{t+bu}{2} \end{pmatrix}\in\Gamma,
\end{equation*}
with $t^{2}-du^{2} = 4$ being a solution of the Pell equation. For $u$ nonzero, $P(t, u)$ is hyperbolic with norm $(t+u \sqrt{d})^{2}/4$ and trace $t$. Because $P(-t,-u)=P(t,u)$ holds in $\Gamma$, we shall restrict~to $t>0$ without loss of generality.
This is in harmony with our convention in Section~\ref{statement-of-main-results} that $\tr(P)>2$ for a hyperbolic element $P\in\Gamma$. If $(t_{d}, u_{d})$ denotes the fundamental solution of the Pell equation, then $P(t_{d}, u_{d})$ is a primitive hyperbolic matrix of norm $\epsilon_{d}^{2}$ and trace $t_{d}$. Moreover, every automorph $P(t,u)$ with $u>0$ (resp.\ $u<0$) is a unique positive (resp.\ negative) integral power of $P(t_d,u_d)$. Sarnak's bijection sends the quadratic form $ax^{2}+bxy+cy^{2}$ to the conjugacy class of $P(t_{d}, u_{d})$ in $\Gamma$. Thus, for a given discriminant $d>0$, there are $h(d)$ primitive hyperbolic conjugacy classes in $\Gamma$, each of norm $\epsilon_{d}^{2}$ and trace $t_{d}$.

Now every hyperbolic conjugacy class $\{P \}$ can be written uniquely as $\{P_{0}^n\}$ for $n \geq 1$ and a primitive hyperbolic conjugacy class $\{P_{0} \}$ (cf.~\cite{Huber1959}). Combining this with Sarnak's bijection described above, we obtain
\begin{equation*}
\Psi_{\Gamma}(x; p, a) = 2 \sum_{\substack{3 \leq t \leq X \\ t \equiv a \pmod p}} 
\sum_{t^{2}-du^{2} = 4} h(d) \log \epsilon_{d},
\end{equation*}
where $X$ abbreviates $\sqrt{x}+\frac{1}{\sqrt{x}}$, and $d>0$ (resp.\ $u>0$) runs through discriminants (resp.\ integers). The class number formula $h(d) \log \epsilon_{d} = \sqrt{d} L(1, \chi_{d})$, where $\chi_{d}$ is the not necessarily primitive quadratic Dirichlet character associated to the discriminant $d$, allows us to write
\begin{equation}\label{eq1}
\Psi_{\Gamma}(x; p, a) = 2 \sum_{\substack{3 \leq t \leq X \\ t \equiv a \pmod p}} \sum_{t^{2}-du^{2} = 4} \sqrt{d} L(1, \chi_{d}).
\end{equation}

For an arbitrary discriminant $\delta>0$, we define Zagier's $L$-series by (cf. \cite[(6) \& (3)]{SoundararajanYoung2013})
\[L(s, \delta)\coloneqq\sum_{du^{2} = \delta}L(s, \chi_{d}) u^{1-2s}=\sum_{q=1}^\infty\frac{\lambda_q(\delta)}{q^s},\]
where $d>0$ (resp.\ $u>0$) runs through discriminants (resp.\ integers). This series admits a transparent Euler product expansion (to be discussed below), while it simplifies \eqref{eq1} as
\begin{equation}\label{eq2}
\Psi_{\Gamma}(x; p, a) = 2 \sum_{\substack{3 \leq t \leq X \\ t \equiv a \pmod{p}}} \sqrt{t^{2}-4} L(1, t^{2}-4).
\end{equation}

If $\delta=Dl^2$, where $D>0$ is a fundamental discriminant and $l>0$ is an integer, then we obtain the following Euler product expansion of Zagier's $L$-series (cf. \cite[(2)]{SoundararajanYoung2013}):
\begin{align}
L(s, \delta)
\notag&=\sum_{u\mid l}L(s,\chi_{Dl^2/u^2})u^{1-2s}\\
\notag&=L(s,\chi_D)\sum_{u\mid l}u^{1-2s}\prod_{\bp\mid\frac{l}{u}}(1-\chi_D(\bp))\\
\label{eq4}&=\prod_{\bp}\left(\sum_{0\leq m<v_{\bp}(l)}\bp^{m(1-2s)}+\frac{\bp^{v_{\bp}(l)(1-2s)}}{1-\chi_D(\bp)\bp^{-s}}\right).
\end{align}
In particular, for fixed $\delta$, the arithmetic function $q\mapsto\lambda_q(\delta)$ is multiplicative. The idea of the proof of Theorem~\ref{main} is to group together certain values of $t$ in \eqref{eq2} such that the corresponding Zagier $L$-series $L(s,t^2-4)$ has a constant Euler factor at $\bp=p$. Thus we are led to consider $L(s,\delta)$ with its Euler factor at $\bp=p$ removed:
\[L^{p}(s, \delta)\coloneqq\sum_{\substack{q\geq 1\\(q,p)=1}}\frac{\lambda_{q}(\delta)}{q^{s}}.\]

\begin{proposition}\label{proposition-2}
Let $p\geq 3$ be a prime, and let $n\geq 1$ be an integer. Let $r\pmod{p^n}$ be an arbitrary residue class. If $(q, p) = 1$ and $b$ denotes the squarefree part of $q$, then
\begin{equation*}
\sum_{\substack{3\leq t \leq X \\ t \equiv r \pmod{p^n}}} \lambda_{q}(t^{2}-4)
 = \frac{X}{p^n} \cdot \frac{\mu(b)}{b}+O_{\epsilon}(q^{\frac{1}{2}+\epsilon}).
\end{equation*}
\end{proposition}

\begin{proof}
It follows from~\cite[Lemma~2.3]{SoundararajanYoung2013} that
\[\lambda_{q}(t^{2}-4) = \sum_{q_{1}^{2} q_{2} = q} \frac{1}{q_{2}} 
\sum_{k \pmod{q_{2}}} e \left(\frac{kt}{q_{2}} \right) S(k^{2}, 1; q_{2}).\]
This leads to
\begin{equation*}
\sum_{\substack{3\leq t \leq X \\ t \equiv r \pmod{p^n}}} \lambda_{q}(t^{2}-4)
 = \sum_{q_{1}^{2} q_{2} = q} \frac{1}{q_{2}} \sum_{k \pmod{q_{2}}} S(k^{2}, 1; q_{2}) 
\sum_{\substack{3\leq t \leq X \\ t \equiv r \pmod{p^n}}} e \left(\frac{kt}{q_{2}} \right).
\end{equation*}
When $k \equiv 0\pmod{q_2}$, the inner sum over $t$ is $\frac{X}{p^n}+O(1)$, and $S(0, 1; q_{2}) = \mu(q_{2})$, yielding the expected main term. When $k \not\equiv 0\pmod{q_2}$, we apply the Weil bound for Kloosterman sums and the estimate
\begin{equation*}
\sum_{\substack{3\leq t \leq X \\ t \equiv r \pmod{p^n}}} e \left(\frac{kt}{q_{2}} \right) \ll \norm{\frac{kp^n}{q_{2}}}^{-1} \leq q_2,
\end{equation*}
where $\norm{\cdot}$ is the distance to the nearest integer. The proof is complete.
\end{proof}
Guided by Proposition~\ref{proposition-2} and \eqref{eq2}, we consider the sum
\begin{equation}\label{eq3}
\Psi_{\Gamma}^{\star}(x; p^n, r) \coloneqq 2 \sum_{\substack{3 \leq t \leq X \\ t \equiv r \pmod{p^n}}} \sqrt{t^{2}-4} L^{p}(1, t^{2}-4).\end{equation}
We shall deduce Theorem~\ref{main} from the following analogue of \cite[Theorem~3.2]{SoundararajanYoung2013}:

\begin{proposition}\label{proposition-3}
Let $p\geq 3$ be a prime, and let $n\geq 1$ be an integer. Let $r\pmod{p^n}$ be an arbitrary residue class. Then
\begin{equation}\label{eq:psidiff}
\Psi_{\Gamma}^{\star}(x+u; p^n, r)-\Psi_{\Gamma}^{\star}(x; p^n, r)
 = \frac{u}{p^n}+O_{\epsilon}(u^{\frac{1}{2}} x^{\frac{1}{4}+\frac{\vartheta}{2}+\epsilon}),\qquad \sqrt{x}\leq u\leq x.
\end{equation}
\end{proposition}
\begin{proof}
Let $\sqrt{x}\leq u\leq x$, and set
\[X \coloneqq \sqrt{x}+\frac{1}{\sqrt{x}}\qquad\text{and}\qquad X^{\prime} \coloneqq \sqrt{x+u}+\frac{1}{\sqrt{x+u}}.\]
From the definition \eqref{eq3}, it is clear that
\begin{align*}
\Psi_{\Gamma}^{\star}(x+u; p^n, r)-\Psi_{\Gamma}^{\star}(x; p^n, r)
& = 2 \sum_{\substack{X < t \leq X^{\prime} \\ t \equiv r \pmod{p^n}}} \sqrt{t^{2}-4} L^{p}(1, t^{2}-4)\\
& = \left(2+O \left(x^{-1} \right) \right) \sum_{\substack{X < t \leq X^{\prime} \\ t \equiv r \pmod{p^n}}} t L^{p}(1, t^{2}-4),
\end{align*}
because $\sqrt{t^{2}-4} = t(1+O(t^{-2}))$. We shall approximate $L^p(1, t^{2}-4)$ in terms of a suitable Dirichlet series. Let $V\geq 1$ be a parameter to be chosen later, and let
\[\delta = t^{2}-4 = D l^{2},\]
where $D>0$ is a fundamental discriminant and $l>0$ is an integer. Consider
\begin{equation*}
\Lambda_{V}^{p}(\delta) \coloneqq \sum_{\substack{q\geq 1\\(q,p)=1}} \frac{\lambda_{q}(\delta)}{q} e^{-\frac{q}{V}}.
\end{equation*}
Shifting the contour yields the expression
\begin{equation*}
\Lambda_{V}^{p}(\delta) = \int_{(1)} L^{p}(1+s, \delta) V^{s} \Gamma(s) \frac{ds}{2\pi i}
 = L^{p}(1, \delta)+\int_{(-\frac{1}{2})} L^{p}(1+s, \delta) V^{s} \Gamma(s) \frac{ds}{2\pi i}.
\end{equation*}
On the right-hand side, for some $A>0$,
\[L^{p}(1+s, \delta)\ll |L(1+s, \delta)|\ll_\epsilon|L(1+s,\chi_D)|l^\epsilon\ll_\epsilon\delta^{\vartheta+\epsilon}|s|^A,\qquad \Re s=-\frac{1}{2},\]
while $\Gamma(s)$ decays exponentially. It follows that
\[L^{p}(1, \delta)=\Lambda_{V}^{p}(\delta)+O_\epsilon( \delta^{\vartheta+\epsilon} V^{-\frac{1}{2}}),\]
and hence
\begin{equation*}
\Psi_{\Gamma}^{\star}(x+u; p^n, r)-\Psi_{\Gamma}^{\star}(x; p^n, r)
 = (2+O(x^{-1})) \sum_{\substack{X < t \leq X^{\prime} \\ t \equiv r \pmod{p^n}}} 
t \Lambda_{V}^{p}(t^{2}-4)+O_{\epsilon}(ux^{\vartheta+\epsilon} V^{-\frac{1}{2}}).
\end{equation*}

If $(q, p) = 1$ and $q = bc^{2}$ with $b$ squarefree, then Proposition~\ref{proposition-2} along with partial summation leads to
\begin{equation*}
2 \sum_{\substack{X < t \leq X^{\prime} \\ t \equiv r \pmod{p^n}}} t \lambda_{q}(t^{2}-4)
 = \frac{u}{p^n}\cdot \frac{\mu(b)}{b}+O_{\epsilon}(Xq^{\frac{1}{2}+\epsilon}).
\end{equation*}
It thus follows that
\begin{equation*}
2 \sum_{\substack{X < t \leq X^{\prime} \\ t \equiv r \pmod{p^n}}} t \Lambda_{V}^{p}(t^{2}-4)
 = \frac{u}{p^n} \sum_{\substack{b, c \geq 1 \\ (bc, p) = 1}} \frac{\mu(b)}{b^{2} c^{2}} e^{-\frac{bc^{2}}{V}}+O_{\epsilon}(XV^{\frac{1}{2}+\epsilon}).
\end{equation*}
A standard contour shift argument gives
\begin{equation*}
\sum_{\substack{b, c \geq 1 \\ (bc, p) = 1}} \frac{\mu(b)}{b^{2} c^{2}} e^{-\frac{bc^{2}}{V}}
 = \int_{(1)} V^{s} \Gamma(s) \frac{\zeta^{p}(2+2s)}{\zeta^{p}(2+s)} \frac{ds}{2\pi i}
 = 1+O(V^{-\frac{1}{2}}),
\end{equation*}
where $\zeta^{p}(s)$ is the Riemann zeta function with the Euler factor at $p$ removed. As a result,
\begin{equation*}
\Psi_{\Gamma}^{\star}(x+u; p^n, r)-\Psi_{\Gamma}^{\star}(x; p^n, r)
 = \frac{u}{p^n}+O_{\epsilon}(x^\frac{1}{2}V^{\frac{1}{2}+\epsilon}+ux^{\vartheta+\epsilon} V^{-\frac{1}{2}}).
\end{equation*}
Setting $V = u x^{-\frac{1}{2}+\vartheta}$ yields \eqref{eq:psidiff}.
\end{proof}

\begin{corollary}\label{corollary-2}
Let $p\geq 3$ be a prime, and let $n\geq 1$ be an integer. Let $r\pmod{p^n}$ be an arbitrary residue class. Then
\begin{equation}\label{eq:removal}
\Psi_{\Gamma}^{\star}(x; p^n, r) = \frac{x}{p^n}+O_{\epsilon}(x^{\frac{3}{4}+\frac{\vartheta}{2}+\epsilon}).
\end{equation}
\end{corollary}

\begin{proof}
Setting $u = x$ in \eqref{eq:psidiff} yields
\begin{equation*}
\Psi_{\Gamma}^{\star}(2x; p^n, r)-\Psi_{\Gamma}^{\star}(x; p^n, r)
 = \frac{x}{p^n}+O_{\epsilon}(x^{\frac{3}{4}+\frac{\vartheta}{2}+\epsilon}).
\end{equation*}
Now \eqref{eq:removal} follows readily by a dyadic subdivision.
\end{proof}

\section{Proof of Theorem~\ref{main}}\label{proof-of-Theorem-1.1} 

We shall approximate the $t$-sum in \eqref{eq2}. As before, we write
\[\delta = t^{2}-4 = D l^{2},\]
where $D>0$ is a fundamental discriminant and $l>0$ is an integer.

If $a\not\equiv\pm 2\pmod{p}$, then for every $t$ participating in \eqref{eq2}, we have that $p\nmid l$ and
\[\chi_D(p)=\left(\frac{D}{p}\right)=\left(\frac{Dl^2}{p}\right)=\left(\frac{t^2-4}{p}\right)=\left(\frac{a^2-4}{p}\right).\]
By \eqref{eq4}, the corresponding Zagier $L$-series $L(s,t^2-4)$ factorises as
\[L(s,t^2-4)=\left(1-\left(\frac{a^{2}-4}{p}\right) p^{-s}\right)^{-1}L^p(s,t^2-4),\]
hence by \eqref{eq2} and \eqref{eq3}, it also follows that
\[\Psi_{\Gamma}(x; p, a) = \left(1-\left(\frac{a^{2}-4}{p}\right) p^{-1}\right)^{-1}\Psi_{\Gamma}^{\star}(x; p, a).\]
Applying \eqref{eq:removal}, we obtain the first two cases of \eqref{eq:asymptotic}.

If $a\equiv\pm 2\pmod{p}$, then we shall assume (without loss of generality) that $a=\pm 2$. We subdivide the $t$-sum in \eqref{eq2} according to the exponent of $p$ in the positive integer $t-a$:
\[\Psi_{\Gamma}(x; p, a)=\sum_{k=1}^\infty\Psi_{\Gamma}(x; p, a; k),\]
where
\[\Psi_{\Gamma}(x; p, a; k)\coloneqq 2\sum_{\substack{3 \leq t \leq X \\ v_p(t-a)=k}} \sqrt{t^{2}-4} L(1, t^{2}-4).\]
We shall approximate these pieces individually. Note that $\Psi_{\Gamma}(x; p, a; k)=0$ for $p^k>t-a$. Moreover, the condition $v_p(t-a)=k$ constrains $t$ to $p-1$ residue classes modulo $p^{k+1}$, and it also yields $v_p(t^2-4)=k$.

If $k=2n-1$ is odd, then $p\mid D$ and $v_p(l)=n-1$, hence by \eqref{eq4},
\[L(s,t^2-4)=\frac{1-p^{n(1-2s)}}{1-p^{1-2s}}L^p(s,t^2-4).\]
Using also \eqref{eq3} and \eqref{eq:removal}, we obtain
\begin{align}
\Psi_{\Gamma}(x; p, a; 2n-1)
\notag&=\frac{p-1}{p^{2n}}\cdot\frac{1-p^{-n}}{1-p^{-1}}\cdot x+O_{p,\epsilon}(x^{\frac{3}{4}+\frac{\vartheta}{2}+\epsilon})\\
\label{eq6}&=(p^{1-2n}-p^{1-3n})x+O_{p,\epsilon}(x^{\frac{3}{4}+\frac{\vartheta}{2}+\epsilon}).
\end{align}
It is important that the implied constant in the error term is independent of $n$.

If $k=2n$ is even, then $p\nmid D$ and $v_p(l)=n$, hence by \eqref{eq4},
\[L(s,t^2-4)=\left(\frac{1-p^{n(1-2s)}}{1-p^{1-2s}}+\frac{p^{n(1-2s)}}{1-\chi_D(p)p^{-s}}\right)L^p(s,t^2-4).\]
We can understand $\chi_D(p)$ by writing $t=a+p^{2n}r$. Indeed, then $t^2-4=2a p^{2n}r+p^{4n} r^2$, hence
\[\chi_D(p)=\left(\frac{D}{p}\right)=\left(\frac{Dl^2p^{-2n}}{p}\right)=\left(\frac{2ar}{p}\right).\]
This means that among the $p-1$ choices for $t\pmod{p^{2n+1}}$, half the time $\chi_D(p)$ equals $1$ and half the time it equals $-1$. Using also \eqref{eq3} and \eqref{eq:removal}, we obtain
\begin{align}
\Psi_{\Gamma}(x; p, a; 2n)
\notag&=\frac{p-1}{p^{2n+1}}\left(\frac{1-p^{-n}}{1-p^{-1}}+\frac{(1/2)p^{-n}}{1-p^{-1}}+\frac{(1/2)p^{-n}}{1+p^{-1}}\right)x+O_{p,\epsilon}(x^{\frac{3}{4}+\frac{\vartheta}{2}+\epsilon})\\
\label{eq7}&=\left(p^{-2n}-\frac{p^{-3n}}{p+1}\right)x+O_{p,\epsilon}(x^{\frac{3}{4}+\frac{\vartheta}{2}+\epsilon}).
\end{align}
Again, the implied constant in the error term is independent of $n$.

Summing up the pieces $\Psi_{\Gamma}(x; p, a; 2n-1)$ and $\Psi_{\Gamma}(x; p, a; 2n)$ for $1\leq n\leq\log(X+2)$, 
and inserting the approximations \eqref{eq6}--\eqref{eq7}, we deduce the asymptotic formula
\[\Psi_{\Gamma}(x; p, \pm 2)=c_px+O_{p,\epsilon}(x^{\frac{3}{4}+\frac{\vartheta}{2}+\epsilon}),\]
where
\[c_p\coloneqq\sum_{n=1}^\infty\left(p^{1-2n}-p^{1-3n}+p^{-2n}-\frac{p^{-3n}}{p+1}\right)=\frac{p}{p^2-1}.\]
This is the third case of \eqref{eq:asymptotic}, and the proof of Theorem~\ref{main} is complete.

\section{Proof of Corollary~\ref{corollary}}
There are $\frac{p-3}{2}$ residues $a \pmod{p}$ such that $a^{2}-4$ is a nonzero quadratic residue. Indeed, this occurs if and only if $a^{2}-4 \equiv b^{2} \pmod{p}$ for some $b \not \equiv 0 \pmod{p}$, which can be written as $(a+b)(a-b) \equiv 4 \pmod{p}$. Making the change of variables $a+b \equiv 2x \pmod{p}$ and $a-b \equiv 2x^{-1} \pmod{p}$, we obtain $a \equiv x+x^{-1} \pmod{p}$ and $b \equiv x-x^{-1} \pmod{p}$ with the condition $x\not\equiv -1,0,1\pmod{p}$. This restriction leaves $\frac{p-3}{2}$ different ways to choose $a \pmod{p}$. Since there are two solutions to $(\frac{a^{2}-4}{p}) = 0$, we conclude that there are $\frac{p-1}{2}$ residues $a \pmod{p}$ such that $a^{2}-4$ is a nonzero quadratic nonresidue. Now Corollary~\ref{corollary} is immediate from Theorem~\ref{main}.


\begin{thebibliography}{BBHM19}

\bibitem[BBCL22]{BalogBiroCherubiniLaaksonen2022}
A.~Balog, A.~Bir{\'{o}}, G.~Cherubini, and N.~Laaksonen, \emph{{B}ykovskii-type
  theorem for the {P}icard manifold}, Int. Math. Res. Not. IMRN \textbf{2022}
  (2022), no.~3, 1893--1921.

\bibitem[BBHM19]{BalogBiroHarcosMaga2019}
A.~Balog, A.~Bir{\'{o}}, G.~Harcos, and P.~Maga, \emph{The prime geodesic
  theorem in square mean}, J. Number Theory \textbf{198} (2019), 239--249,
  Errata: \url{https://users.renyi.hu/~gharcos/primegeodesic_errata.txt}.

\bibitem[BCC{\etalchar{+}}19]{BalkanovaChatzakosCherubiniFrolenkovLaaksonen2019}
O.~Balkanova, D.~Chatzakos, G.~Cherubini, D.~Frolenkov, and N.~Laaksonen,
  \emph{Prime geodesic theorem in the 3-dimensional hyperbolic space}, Trans.
  Amer. Math. Soc. \textbf{372} (2019), no.~8, 5355--5374.

\bibitem[BF19]{BalkanovaFrolenkov2019}
O.~Balkanova and D.~Frolenkov, \emph{Sums of {K}loosterman sums in the prime
  geodesic theorem}, Quart. J. Math. \textbf{70} (2019), no.~2, 649--674.

\bibitem[BF20]{BalkanovaFrolenkov2020}
\bysame, \emph{Prime geodesic theorem for the {P}icard manifold}, Adv. Math.
  \textbf{375} (2020), Paper No. 107377, 42 pages.

\bibitem[BFR22]{BalkanovaFrolenkovRisager2022}
O.~Balkanova, D.~Frolenkov, and M.~S. Risager, \emph{Prime geodesics and
  averages of the {Z}agier {$L$}-series}, Math. Proc. Cambridge Philos. Soc.
  \textbf{172} (2022), no.~3, 705--728.

\bibitem[Byk94]{Bykovskii1994}
V.~A. Bykovski{\u{\i}}, \emph{Density theorems and the mean value of arithmetic
  functions in short intervals ({R}ussian)}, Zap. Nauchn. Sem. S.-Peterburg.
  Otdel. Mat. Inst. Steklov. (POMI) \textbf{212} (1994), no.~Anal. Teor. Chisel
  i Teor. Funktsi{\u{\i}}. 12, 56--70, 196, translation in J. Math. Sci. (N.Y.)
  \textbf{83} (1997), no. 6, 720--730.

\bibitem[Cai02]{Cai2002}
Y.~Cai, \emph{Prime geodesic theorem}, J. Th{\'{e}}or. Nombres Bordeaux
  \textbf{14} (2002), no.~2, 59--72.

\bibitem[CCL22]{ChatzakosCherubiniLaaksonen2022}
D.~Chatzakos, G.~Cherubini, and N.~Laaksonen, \emph{Second moment of the prime
  geodesic theorem for {$\mathrm{PSL}(2, \mathbb{Z}[i])$}}, Math. Z.
  \textbf{300} (2022), no.~1, 791--806.

\bibitem[CG18]{CherubiniGuerreiro2018}
G.~Cherubini and J.~Guerreiro, \emph{Mean square in the prime geodesic
  theorem}, Algebra \& Number Theory \textbf{12} (2018), no.~3, 571--597.

\bibitem[CI00]{ConreyIwaniec2000}
J.~B. Conrey and H.~Iwaniec, \emph{The cubic moment of central values of
  automorphic {$L$}-functions}, Ann. of Math. (2) \textbf{151} (2000), no.~3,
  1175--1216.

\bibitem[CWZ22]{CherubiniWuZabradi2022}
G.~Cherubini, H.~Wu, and G.~Z{\'{a}}br{\'{a}}di, \emph{On
  {K}uznetsov--{B}ykovskii's formula of counting prime geodesics}, Math. Z.
  \textbf{300} (2022), no.~1, 881--928.

\bibitem[DM23]{DeverMilicevic2023}
L.~Dever and D.~Mili{\'{c}}evi{\'{c}}, \emph{Ambient prime geodesic theorems on
  hyperbolic {$3$}-manifolds}, Int. Math. Res. Not. IMRN \textbf{2023} (2023),
  no.~1, 588--635.

\bibitem[GS99]{GolovchanskiiSmotrov1999}
V.~V. Golovchanski{\u{\i}} and M.~N. Smotrov, \emph{The distribution of the
  number classes of the primitive hyperbolic elements of the group
  {$\Gamma_{0}(N)$} in arithmetic progressions ({R}ussian)}, Preprint FEB RAS
  Khabarovsk Division of the Institute for Applied Mathematics (1999), no.~15,
  15 pages,
  \url{https://drive.google.com/file/d/1MxrJbbhLTGKzh4-1lCWa4Up_EDGeaqHc/view}.

\bibitem[Hej76a]{Hejhal1976-2}
D.~Hejhal, \emph{The {S}elberg trace formula and the {R}iemann zeta function},
  Duke Math. J. \textbf{43} (1976), no.~3, 441--482.

\bibitem[Hej76b]{Hejhal1976}
\bysame, \emph{The {S}elberg trace formula for {$\mathrm{PSL}(2, \mathbb{R})$
  I}}, Lecture Notes in Mathematics, vol. 548, Springer-Verlag, Berlin
  Heidelberg, 1976.

\bibitem[Hej83]{Hejhal1983}
\bysame, \emph{The {S}elberg trace formula for {$\mathrm{PSL}(2, \mathbb{R})$
  I\hspace{-.1mm}I}}, Lecture Notes in Mathematics, vol. 1001, Springer-Verlag,
  Berlin Heidelberg, 1983.

\bibitem[Hub59]{Huber1959}
H.~Huber, \emph{{Z}ur analytischen {T}heorie hyperbolischer {R}aumformen und
  {B}ewegungsgruppen}, Math. Ann. \textbf{138} (1959), no.~1, 1--26.

\bibitem[Hub61a]{Huber1961}
\bysame, \emph{{Z}ur analytischen {T}heorie hyperbolischer {R}aumformen und
  {B}ewegungsgruppen. {I\hspace{-.1mm}I}}, Math. Ann. \textbf{142} (1961),
  no.~4, 385--398.

\bibitem[Hub61b]{Huber1961-2}
\bysame, \emph{{Z}ur analytischen {T}heorie hyperbolischer {R}aumformen und
  {B}ewegungsgruppen. {I\hspace{-.1mm}I}. {N}achtrag zu {M}ath. {A}nn.
  {\textbf{142}}, 385--398 (1961)}, Math. Ann. \textbf{143} (1961), no.~5,
  463--464.

\bibitem[Iwa84a]{Iwaniec1984-2}
H.~Iwaniec, \emph{Non-holomorphic modular forms and their applications},
  Modular forms (Durham, 1983) (R.~A. Rankin, ed.), Ellis Horwood Ser. Math.
  Appl., Statist. Oper. Res., Horwood, Chichester, 1984, pp.~157--196.

\bibitem[Iwa84b]{Iwaniec1984}
\bysame, \emph{Prime geodesic theorem}, J. Reine Angew. Math. \textbf{349}
  (1984), 136--159.

\bibitem[Kan20]{Kaneko2020}
I.~Kaneko, \emph{The second moment for counting prime geodesics}, Proc. Japan
  Acad. Ser. A Math. Sci. \textbf{96} (2020), no.~1, 7--12.

\bibitem[Kan22]{Kaneko2022-2}
\bysame, \emph{The prime geodesic theorem for
  {$\mathrm{PSL}_{2}(\mathbb{Z}[i])$} and spectral exponential sums}, Algebra
  Number Theory \textbf{16} (2022), no.~8, 1845--1887.

\bibitem[Kan23]{Kaneko2023-2}
\bysame, \emph{Spectral exponential sums on hyperbolic surfaces}, Ramanujan J.
  \textbf{60} (2023), no.~3, 837--846.

\bibitem[KK22]{KanekoKoyama2022}
I.~Kaneko and S.~Koyama, \emph{Euler products of {S}elberg zeta functions in
  the critical strip}, Ramanujan J. \textbf{59} (2022), no.~2, 437--458.

\bibitem[Koy98]{Koyama1998}
S.~Koyama, \emph{Prime geodesic theorem for arithmetic compact surfaces}, Int.
  Math. Res. Not. IMRN \textbf{1998} (1998), no.~8, 383--388.

\bibitem[Koy01]{Koyama2001}
\bysame, \emph{Prime geodesic theorem for the {P}icard manifold under the
  mean-{L}indel{\"{o}}f hypothesis}, Forum Math. \textbf{13} (2001), no.~6,
  781--793.

\bibitem[Kuz78]{Kuznetsov1978}
N.~V. Kuznetsov, \emph{The arithmetic form of {S}elberg's trace formula and the
  distribution of norms of the primitive hyperbolic classes of the modular
  group}, Preprint (Khabarovsk), 1978.

\bibitem[LRS95]{LuoRudnickSarnak1995}
W.~Luo, Z.~Rudnick, and P.~Sarnak, \emph{On {S}elberg's eigenvalue conjecture},
  Geom. Funct. Anal. \textbf{5} (1995), no.~2, 387--401.

\bibitem[LS95]{LuoSarnak1995}
W.~Luo and P.~Sarnak, \emph{Quantum ergodicity of eigenfunctions on
  {$\mathrm{PSL}_{2}(\mathbf{Z}) \backslash \mathbf{H}^{2}$}}, Publ. Math.
  Inst. Hautes {\'{E}}tudes Sci. \textbf{81} (1995), 207--237.

\bibitem[PR17]{PetridisRisager2017}
Y.~N. Petridis and M.~S. Risager, \emph{Local average in hyperbolic lattice
  point counting, with an appendix by {N}iko {L}aaksonen}, Math. Z.
  \textbf{285} (2017), no.~3--4, 1319--1344.

\bibitem[Sar80]{Sarnak1980}
P.~Sarnak, \emph{Prime geodesic theorems}, {Ph.D.} thesis, Stanford University,
  August 1980.

\bibitem[Sar82]{Sarnak1982}
\bysame, \emph{Class numbers of indefinite binary quadratic forms}, J. Number
  Theory \textbf{15} (1982), no.~2, 229--247.

\bibitem[Sel89]{Selberg1989}
A.~Selberg, \emph{Collected papers}, Springer Collected Works in Mathematics,
  vol.~1, Springer-Verlag Berlin Heidelberg, 1989.

\bibitem[SY13]{SoundararajanYoung2013}
K.~Soundararajan and M.~P. Young, \emph{The prime geodesic theorem}, J. Reine
  Angew. Math. \textbf{676} (2013), 105--120, Errata:
  \url{https://www.math.tamu.edu/\%7Emyoung/PrimeGeodesicCorrections.pdf}.

\bibitem[Ven90]{Venkov1990}
A.~B. Venkov, \emph{Spectral theory of automorphic functions and its
  applications}, Mathematics and its Applications (Soviet Series), vol.~51,
  Kluwer Academic Publishers Group, 1990, Translated from the Russian by N. B.
  Lebedinskaya.

\end{thebibliography}

\newcommand{\etalchar}[1]{$^{#1}$}

\end{document}